\documentclass[english]{elsarticle}
\usepackage[T1]{fontenc}
\usepackage[latin9]{inputenc}
\usepackage{babel}
\usepackage{varioref}
\usepackage{prettyref}
\usepackage{float}
\usepackage{amsmath}
\usepackage{amsthm}
\usepackage{amssymb}
\usepackage{mathdots}
\usepackage{graphicx}
\usepackage{xargs}[2008/03/08]
\usepackage[unicode=true,
 bookmarks=true,bookmarksnumbered=false,bookmarksopen=true,bookmarksopenlevel=1,
 breaklinks=false,pdfborder={0 0 1},backref=false,colorlinks=false]
 {hyperref}
\usepackage{breakurl}

\makeatletter

\providecommand{\tabularnewline}{\\}
\floatstyle{ruled}
\newfloat{algorithm}{tbp}{loa}
\providecommand{\algorithmname}{Algorithm}
\floatname{algorithm}{\protect\algorithmname}

  \theoremstyle{definition}
  \newtheorem{problem}{\protect\problemname}
  \theoremstyle{definition}
  \newtheorem{defn}{\protect\definitionname}
\theoremstyle{plain}
\newtheorem{thm}{\protect\theoremname}
  \theoremstyle{remark}
  \newtheorem{rem}{\protect\remarkname}
 \theoremstyle{definition}
  \newtheorem{example}{\protect\examplename}
  \theoremstyle{plain}
  \newtheorem{lem}{\protect\lemmaname}
  \theoremstyle{plain}
  \newtheorem{cor}{\protect\corollaryname}

\usepackage{etex}

\usepackage{bm}

\renewcommand{\vec}[1]{\boldsymbol{#1}}

\newcommand{\acc}{\ensuremath{\varepsilon}}
\newcommand{\er}{\epsilon}
\newcommand{\np}{\ensuremath{s}}
\newcommand{\dg}{\ensuremath{d}}
\newcommand{\nm}{\ensuremath{N}}
\newcommand{\nd}{\ensuremath{z}}
\newcommand{\jc}{\ensuremath{a}}
\newcommand{\jp}{\nd}
\newcommand{\nparams}{R}

\newcommand{\mm}{\tilde{m}}

\newcommand{\df}{p}

\newrefformat{subsec}{Subsection \ref{#1}}
\newrefformat{prob}{Problem \ref{#1}}
\newrefformat{alg}{Algorithm \ref{#1}}
\newrefformat{rem}{Remark \ref{#1}}
\newrefformat{def}{Definition \ref{#1}}
\newrefformat{tbl}{Table \ref{#1}}

\DeclareMathOperator{\argmin}{argmin}
\DeclareMathOperator{\shift}{E}
\DeclareMathOperator{\id}{I}
\newcommand{\dl}{\Delta}

\DeclareMathOperator{\diag}{diag}

\@ifundefined{showcaptionsetup}{}{%
 \PassOptionsToPackage{caption=false}{subfig}}
\usepackage{subfig}
\makeatother

  \providecommand{\definitionname}{Definition}
  \providecommand{\examplename}{Example}
  \providecommand{\lemmaname}{Lemma}
  \providecommand{\problemname}{Problem}
  \providecommand{\remarkname}{Remark}
\providecommand{\corollaryname}{Corollary}
\providecommand{\theoremname}{Theorem}

\begin{document}

\begin{frontmatter}{}

\title{Accurate solution of near-colliding Prony systems via decimation and homotopy continuation}

\author{Dmitry Batenkov\fnref{curraddr,support}}

\address{Department of Computer Science, Technion - Israel Institute of Technology, Haifa 32000, Israel}

\fntext[curraddr]{Current address: Department of Mathematics, Massachusetts Institute of Technology,
Cambridge, MA 02139, USA. email: batenkov@mit.edu}

\fntext[support]{The research leading to these results has received funding from the
European Research Council under European Union\textquoteright s Seventh
Framework Program, ERC Grant agreement no. 320649.}

\ead{batenkov@cs.technion.ac.il}
\begin{abstract}
We consider polynomial systems of Prony type, appearing in many areas
of mathematics. Their robust numerical solution is considered to be
difficult, especially in ``near-colliding'' situations. We consider
a case when the structure of the system is a-priori fixed. We transform
the nonlinear part of the Prony system into a Hankel-type polynomial
system. Combining this representation with a recently discovered ``decimation''
technique, we present an algorithm which applies homotopy continuation
to an appropriately chosen Hankel-type system as above. In this way,
we are able to solve for the nonlinear variables of the original system
with high accuracy when the data is perturbed.
\end{abstract}

\end{frontmatter}{}

\global\long\def\CC{\mathbb{C}}

\global\long\def\RR{\mathbb{R}}

\global\long\def\NN{\mathbb{N}}

\global\long\def\TT{\mathbb{T}}

\newcommandx\meas[1][usedefault, addprefix=\global, 1=k]{m_{#1}}

\newcommandx\fwm[1][usedefault, addprefix=\global, 1=\nm]{{\cal P_{#1}}}
\newcommandx\jac[1][usedefault, addprefix=\global, 1=\nm]{{\cal J_{#1}}}
\newcommandx\pinv[1][usedefault, addprefix=\global, 1=\nm]{\jac[#1]^{\dagger}}

\global\long\def\isdef{\ensuremath{:=}}

\global\long\def\decmap{{\cal P}^{\left(\df\right)}}
\global\long\def\decjac{{\cal J}^{\left(\df\right)}}
\newcommandx\scmap[1][usedefault, addprefix=\global, 1=\df]{{\cal R}_{#1}}

\global\long\def\ddd{\ensuremath{\mathfrak{D}}}

\section{Introduction}

\subsection{The Prony problem}

Consider the following approximate algebraic problem.
\begin{problem}
\label{prob:prony}Given $\left(\mm_{0},\dots,\mm_{\nm-1}\right)\in\CC^{\nm}$,
$\acc\geqslant0$, $\np\in\NN$, a multiplicity vector $D=\left(d_{1},\dots,d_{\np}\right)\in\NN^{\np}$
with $\dg:=\sum_{j=1}^{\np}d_{j}$ and $2d\leqslant\nm$, find complex
numbers $\{\nd_{j},\{\jc_{\ell,j}\}_{\ell=0}^{d_{j}-1}\}_{j=1}^{\np}$
satisfying $a_{d_{j}-1,j}\neq0$ and $\left|\nd_{j}\right|=1$ with
$\left\{ \nd_{j}\right\} $ pairwise distinct, such that for some
perturbation vector $\left(\er_{0},\dots,\er_{\nm-1}\right)\in\CC^{\nm}$
with $|\er_{k}|<\acc$ we have
\begin{equation}
\mm_{k}=\underbrace{\sum_{j=1}^{\np}\nd_{j}^{k}\sum_{\ell=0}^{d_{j}-1}\jc_{\ell,j}k^{\ell}}_{:=m_{k}}+\er_{k},\qquad k=0,\dots,\nm-1.\label{eq:prony}
\end{equation}
\end{problem}
This so-called \emph{confluent Prony problem} and its numerous variations
appear in signal processing, frequency estimation, exponential fitting,
Pad\'{e} approximation, sparse polynomial interpolation, spectral
edge detection, inverse moment problems and recently in theory of
super-resolution (see \cite{badeau2006high,batFullFourier,byPronySing12,ben-or_deterministic_1988,candes2012towards,comer_sparse_2012,eckhoff1995arf,giesbrecht_symbolicnumeric_2009,pereyra_exponential_2010,potts2010parameter}
and references therein). We comment on the specific assumptions made
in the above formulation in \prettyref{subsec:Assumptions} below.

Several solution methods for Prony systems have been proposed in the
literature, starting with the classical Prony's method and including
algorithms such as MUSIC/ESPRIT, Approximate Prony Method, matrix
pencils, Total Least Squares, Variable Projections (VARPRO) or $\ell_{1}$
minimization (\cite{batenkov2011accuracy,beckermann_numerical_2007,cadzow_total_1994,candes2012towards,oleary_variable_2013,lee_quotient-difference_2007,potts2010parameter}
and references therein). While the majority of these algorithms perform
well on simple and well-separated nodes, they are somewhat poorly
adapted to handle either multiple/clustered nodes (the root extraction/eigenvalue
computation becoming ill-conditioned), large values of $\nm$ (the
quadratic cost function is highly non-convex w.r.t to the unknowns
$\nd_{j},\jc_{\ell,j}$) or non-Gaussian noise. Despite this, our
recent studies \cite{batFullFourier,batenkov_numerical_2014,batenkov2011accuracy}
suggest that these problems are only partially due to the inherent
sensitivity of the problem (i.e. \emph{problem conditioning}). Generally
speaking, introduction of confluent (high-order) nodes into the model
leads, in some cases, to improved estimation of the parameters \textendash{}
as indicated by the reduced condition number of the problem. In particular,
we argue that while for $\nm\delta\gg1$, where $\delta$ is the node
separation (see \prettyref{def:separation} below), and $D=\left(1,1,\dots,1\right)$
the existing methods might be close to optimal, there is a gap between
theory and practice in the \textquotedblleft near-collision\textquotedblright{}
situation $\nm\delta\ll1$ and high multiplicity, even if the noise
is sufficiently small.

\emph{Decimation }is a particular regularization for near-colliding Prony systems,
which was first proposed in \cite{batFullFourier} and further analyzed
in \cite{batenkov_numerical_2014}. The essential idea is 
that if $\mbox{\nm\ensuremath{\delta\ll}1}$, then after taking the
decimated sequences $\{m_{0},m_{\df},m_{2\df},\dots,m_{(\nparams-1)\df}\}$,
where $\df\in\NN$ is not too large, as measurements, and solving
the resulting square system, we can get accuracy improvement of the order
of $\df^{-d_{j}}$ for the node $\nd_{j}$ - compared to the error
in the case $\df=1$. Numerical studies carried out in \cite{batenkov_numerical_2014} (see \prettyref{subsec:the-preprint} below for further
details)
indicate that in this case, the best possible resulting accuracy is very close to
the accuracy obtained by least squares - as quantified by the ``near-collision condition
number''. The work \cite{batenkov_numerical_2014} does not suggest any practical solution method which achieves the above accuracy, and in the present paper we propose to fill this gap.

\subsection{Our contribution}

In this paper we focus on developing an accurate solution method for
\prettyref{prob:prony} in the near-collision regime $\nm\delta\ll1$,
in the case of a single cluster.

We propose a novel symbolic-numeric technique, ``decimated homotopy'',
for this task. The approach is an extension of the method used in
\cite{batFullFourier} for the case $\np=1$ (i.e.only one node),
and its main ingredients are:
\begin{enumerate}
\item decimating the measurements;
\item constructing a square polynomial system for the unknowns $\left\{ \nd_{j}\right\} $
(in \cite{batFullFourier} this was a single polynomial equation)
;
\item solving the resulting \emph{well-conditioned} system with high accuracy;
\item pruning the spurious solutions and recovering the solution to the
original system.
\end{enumerate}
Step 2 above is a purely symbolic computation based on the structure
of the equations \eqref{eq:prony}, while for step 3 we chose the
homotopy continuation method for polynomial systems, due to the fact
that it will provably find the solution. We propose several alternatives
for the pruning step 4, and discuss their efficiency.

We also show that the proposed algorithm recovers the nodes with high
accuracy \textendash{} in fact, with near-optimal accuracy. Numerical
simulations demonstrate that the algorithm is accurate as predicted,
and outperforms ESPRIT in this setting.

Some of the presented results have been published in abstract form in the proceedings of
the SNC'14 meeting \cite{batenkov_prony_2014}, which took place in
Shanghai, China, during July 2014. 

\subsection{\label{subsec:Assumptions}More on the assumptions}

Let us briefly comment on the specific assumptions made throughout
the paper, and point out some current limitations.
\begin{enumerate}
\item The setting $\left|\nd_{j}\right|=1$ is common in applied harmonic
analysis, where the prototype model for \prettyref{prob:prony} is
to recover a Dirac measure $f\left(x\right)=\sum_{i=1}^{\np}\jc_{i}\delta\left(x-x_{i}\right)$
from the Fourier coefficients $c_{k}\left(f\right)=\frac{1}{2\pi}\int e^{-\imath kt}f\left(t\right)dt$.
Dropping this assumption will in general have severe consequences
in terms of numerical stability of the problem.
\item The confluent/high-multiplicity models are also quite common in inverse
problems involving some sort of derivatives. In the most general formulation,
the multiplicity structure $D$ may be unknown, and it is an important
question to determine it reliably from the data and other a-priori
information. This is an ongoing research effort, and we hope that
the methods of the present paper may serve as a building block for
this goal. See also a discussion in \prettyref{sec:future-work} below.
\item Similarly, extending the treatment to multiple clusters is an ongoing
work. In this regard, our method can be regarded as a \emph{zooming}
technique.
\item The noise level $\er$ is assumed to be sufficiently small. Quantification
of the allowed noise level for the problems to be solvable is a closely
related question and treated elsewhere, see e.g. \cite{akinshin_accuracy_2015,byPronySing12}.
\end{enumerate}

\subsection{Organization of the paper}

In \prettyref{sec:prior-work} we discuss in detail the relevant prior
work, in particular accuracy bounds on (decimated) Prony systems \cite{batenkov2011accuracy,batenkov_numerical_2014}
and the algebraic reconstruction method for the case $\np=1$ from
\cite{batFullFourier,batenkov_algebraic_2012}. The decimated homotopy
algorithm is subsequently developed in \prettyref{sec:decimated-homotopy-algorithm}.
Analysis of the algorithm and its accuracy is presented in \prettyref{sec:Analysis}.
Results of numerical experiments are described in \prettyref{sec:numerical-experiments},
while several future research directions are outlined in \prettyref{sec:future-work}.

\section{\label{sec:prior-work}Accuracy of solving Prony systems, decimation
and algebraic reconstruction}

We start with brief recap of the Prony's method in \prettyref{subsec:Prony-method}.
Following \cite{batenkov_numerical_2014}, in \prettyref{subsec:the-preprint}
we present numerical stability bounds, including in the decimated
scenario, for the system \eqref{eq:prony}. In \prettyref{subsec:one-point}
we discuss the ``algebraic'' reconstruction algorithm for the system
\eqref{eq:prony} with $\np=1$, used in \cite{batFullFourier,batenkov_algebraic_2012},
and highlight some of its key properties, in particular the effect
of decimation on its accuracy. 

\subsection{\label{subsec:Prony-method}Prony's method}

The high degree of symmetry in the system of equations \eqref{eq:prony}
allows to separate the problem into a linear and a nonlinear part.
The basic observation (due to R. de Prony \cite{prony1795essai})
is that the sequence of exact measurements $\{m_{k}\}$ satisfies
a linear recurrence relation
\begin{equation}
\sum_{\ell=0}^{\dg}m_{k+\ell}c_{\ell}=0,\qquad k\in\NN,\label{eq:rec-rel}
\end{equation}
where $\left\{ c_{\ell}\right\} $ are the coefficients of the \emph{Prony
polynomial} defined as
\begin{equation}
P\left(x\right)\isdef\prod_{j=1}^{\np}(x-\nd_{j})^{d_{j}}\equiv\sum_{\ell=0}^{\dg}c_{\ell}x^{\ell}.\label{eq:prony-polynomial}
\end{equation}
Thus, the system \eqref{eq:prony} can be solved for $\nm=2d$ by
the following steps.
\begin{enumerate}
\item Using \eqref{eq:rec-rel}, recover the coefficients $\left\{ c_{\ell}\right\} $
of $P\left(x\right)$ from a non-trivial vector in the nullspace of
the Hankel matrix
\begin{equation}
H_{\dg}\isdef\begin{pmatrix}\mm_{0} & \cdots & \mm_{d-1} & \mm_{d}\\
\iddots & \iddots & \iddots & \vdots\\
\mm_{d-1} & \mm_{d} & \cdots & \mm_{2d-1}
\end{pmatrix}.\label{eq:hankel-data-matrix}
\end{equation}
\item Recover the nodes $\left\{ \nd_{j}\right\} $ by finding the roots
of $\tilde{P}\left(x\right)$ with appropriate multiplicities.
\item Given the the nodes $\left\{ \nd_{j}\right\} $, recover the coefficients
$\left\{ \jc_{\ell,j}\right\} $ by solving a Vandermonde linear system.
\end{enumerate}

\subsection{\label{subsec:the-preprint}Stability bounds and decimation}

Let us first introduce some notation. The number of unknown parameters
is denoted by $\nparams:=\dg+\np$.
\begin{defn}
The \emph{data space} associated to the Prony problem is the a-priori
set of possible solutions
\begin{equation}
\ddd:=\left\{ \left(\jc_{0,1},\dots,\jc_{d_{1}-1,1},\jp_{1},\dots,\jc_{0,\np},\dots,\jc_{d_{\np}-1,\np},\jp_{\np}\right)^{T}\in\CC^{\nparams}:\;\left|\jp_{j}\right|=1\right\} .\label{eq:data-point}
\end{equation}
\end{defn}
We also extensively use the notion of node separation, defined as
follows.
\begin{defn}
\label{def:separation}Let $\vec{x}\in\ddd$ be a data point as in
\eqref{eq:data-point}. For $i\neq j$, let $\delta_{ij}:=\left|\arg z_{i}-\arg z_{j}\right|$
with the convention that $\delta_{ij}\leqslant\pi$. For $i=1,\dots,\np$
the \emph{$i$-th node separation} of $\vec{x}$ is 
\begin{equation}
\delta^{\left(i\right)}=\delta^{\left(i\right)}\left(\vec{x}\right):=\min_{i\neq j}\delta_{ij}.\label{eq:separation-def}
\end{equation}
In addition, we denote the global separation as
\[
\delta=\delta\left(\vec{x}\right):=\min_{i}\delta^{\left(i\right)}.
\]

For any $\nm\geqslant\nparams$, let the \emph{forward mapping} $\fwm:\ddd\to\CC^{\nm}$
be given by the measurements, i.e. for any $\vec{x}\in\ddd$ (see
\eqref{eq:data-point}) we have
\[
\fwm\left(\vec{x}\right):=\left(\meas[0],\dots,\meas[\nm-1]\right)^{T},
\]
where $\meas$ are given by \eqref{eq:prony}.
\end{defn}
A standard measure of sensitivity \cite{dayton_multiple_2011,stetter_numerical_2004}
for well-conditioned polynomial systems is the following.
\begin{defn}
Let $\vec{x}\in\ddd$ be a point in the data space. Assume that $\jac\left(\vec{x}\right)\isdef d\fwm\left(\vec{x}\right)$,
the Jacobian matrix of the mapping $\fwm$ at the point $\vec{x}$,
has full rank. For $\alpha=1,2,\dots,\nparams,$ the \emph{component-wise
condition number }of parameter $\alpha$ at the data point $\vec{x}$
is the quantity
\begin{equation}
CN_{\alpha,\nm}\left(\vec{x}\right):=\sum_{i=1}^{\nm}\left|\pinv\left(\vec{x}\right)_{\alpha,i}\right|,\label{eq:cn-full-def}
\end{equation}
where $\pinv$ is the Moore-Penrose pseudo-inverse of $\jac$.
\end{defn}
In \cite{batenkov_numerical_2014} we show that for $\nm\delta\gg1$,
the Prony system \eqref{eq:prony} is well-conditioned as follows (up to changes of notation and a slightly different noise model, see footnote on page 4 in \cite{batenkov_numerical_2014}):
\begin{thm}[Theorem 2.1 in \cite{batenkov_numerical_2014}]
\label{thm:stability-1}Let $\vec{x}\in\ddd$ be a data point, such
that $\delta=\delta\left(\vec{x}\right)>0$ and $\jc_{d_{j}-1,j}\neq0$
for $j=1,\dots,\np$. Then

\begin{enumerate}
\item The Jacobian matrix $\jac\left(\vec{x}\right)=d\fwm\left(\vec{x}\right)\in\CC^{\nm\times\nparams}$
has full rank.
\item There exist constants $K$, $C^{\left(1\right)}$, not depending on
$\nm$ and $\delta$, such that for $\nm>K\cdot\delta^{-1}$:
\begin{eqnarray}
CN_{z_{j},\nm}\left(\vec{x}\right) & \leqslant & C^{\left(1\right)}\cdot\frac{1}{\left|\jc_{d_{j}-1,j}\right|}.\frac{1}{\nm^{d_{j}}}.\label{eq:good-stability}
\end{eqnarray}
\end{enumerate}
\end{thm}

It is easy to show (see e.g. \cite[Appendix A]{batFullFourier}) that
the upper bound on $CN_{\nd_{j},\nm}$ is asymptotically tight.

On the other hand, as numerical experiments in \cite{batenkov_numerical_2014}
show, when $\nm\delta\to0$ then the growth of $CN_{\nd_{j},\nm}$
is much more rapid than $\nm^{d_{j}}$ (obviously  $CN_{\nd_{j},\nm}\to\infty$ as the system becomes singular). As we now argue, this ``phase transition'' near $\nm\delta\sim O\left(1\right)$
can be partially quantified by considering a sequence of decimated
square systems.

Fixing $\nm=\nparams$, we have the following upper bound, which is
tight.
\begin{thm}[Theorem 2.2 in \cite{batenkov_numerical_2014}]
\label{thm:stability-2}Assume the conditions of \prettyref{thm:stability-1},
and furthermore that $\nm=\nparams$. Then there exists a constant
$C^{\left(2\right)}$, not depending on $\vec{x}$ (and in particular
on $\delta$), such that:
\begin{eqnarray}
CN_{z_{j},\nparams}\left(\vec{x}\right) & \leqslant & C^{\left(2\right)}\cdot\left(\frac{1}{\delta^{\left(j\right)}}\right)^{\nparams-d_{j}}\cdot\frac{1}{\left|\jc_{d_{j}-1,j}\right|}.\label{eq:square-stability}
\end{eqnarray}
\end{thm}
A natural question is whether increasing $\nm$ can essentially improve
the bound \eqref{eq:square-stability} above. One possible answer
is given by what we call ``decimation'', as follows.
\begin{defn}
Let $\df\in\NN$ be a positive integer. The \emph{decimated Prony
system with parameter $\df$ }is given by
\begin{equation}
n_{k}\isdef\meas[\df k]=\sum_{j=1}^{\np}\jp_{j}^{\df k}\sum_{\ell=0}^{d_{j}-1}\left(\jc_{\ell,j}\df^{\ell}\right)k^{\ell},\quad k=0,1,\dots,\nparams-1.\label{eq:decimated-prony-def}
\end{equation}
\end{defn}
\begin{minipage}[t]{1\columnwidth}%
\end{minipage}
\begin{defn}
The \emph{decimated forward map} $\decmap:\CC^{\nparams}\to\CC^{\nparams}$
is given by
\[
\decmap\left(\vec{x}\right):=\left(n_{0},\dots,n_{\nparams-1}\right),
\]
where $\vec{x}\in\ddd$ is as in \eqref{eq:data-point} and $n_{k}$
are given by \eqref{eq:decimated-prony-def}.
\end{defn}
\begin{minipage}[t]{1\columnwidth}%
\end{minipage}
\begin{defn}
The decimated condition numbers $CN_{\alpha}^{\left(\df\right)}$
are defined as 
\begin{equation}
CN_{\alpha}^{\left(\df\right)}\left(\vec{x}\right):=\sum_{i=1}^{\nparams}\left|\left(\left\{ \decjac\left(\vec{x}\right)\right\} ^{-1}\right)_{\alpha,i}\right|,\label{eq:cn-dec-def}
\end{equation}
where $\decjac\left(\vec{x}\right)$ is the Jacobian of the decimated
map $\decmap$ (the definition applies at every point $\vec{x}$ where
the Jacobian is non-degenerate).
\end{defn}
The usefulness of decimation becomes clear given the following result.
\begin{thm}[Corollary 3.1 in \cite{batenkov_numerical_2014}]
\label{thm:decimated-accuracy}Assume the conditions of \prettyref{thm:stability-1}.
Assume further that $\mbox{\nm\ensuremath{\delta^{*}}<\ensuremath{\pi}\nparams}$
where $\delta^{*}:=\max_{i\neq j}\delta_{ij}$ (i.e. all nodes form
a cluster). Then the condition numbers of the decimated system \eqref{eq:decimated-prony-def}
with parameter $\df^{*}:=\left\lfloor \frac{\nm}{\nparams}\right\rfloor $
satisfy
\begin{eqnarray}
CN_{z_{j}}^{\left(\df^{*}\right)}\left(\vec{x}\right) & \leqslant & C^{\left(3\right)}\cdot\left(\frac{1}{\delta^{\left(j\right)}}\right)^{\nparams-d_{j}}\frac{1}{\left|\jc_{d_{j}-1,j}\right|}\cdot\frac{1}{\nm^{\nparams}}.\label{eq:decimated-stability}
\end{eqnarray}
\end{thm}
The intuition behind this result is that decimation with parameter
$\df$ is in fact equivalent to applying the Prony mapping $\fwm[\nparams]$
to a rescaled data point $\mbox{\ensuremath{\vec{y}}:=\ensuremath{\scmap}{\ensuremath{\left(\vec{x}\right)}}}$,
where
\begin{eqnarray}
\begin{split}\scmap\left(\left(\jc_{0,1},\dots,\jc_{d_{1}-1,1},\jp_{1},\dots,\jc_{0,\np},\dots,\jc_{d_{\np}-1,\np},\jp_{\np}\right)^{T}\right) & \isdef\\
\left(b_{0,1},\dots,b_{d_{1}-1,1},w_{1},\dots,b_{0,\np},\dots,b_{d_{\np}-1,\np},w_{\np}\right)^{T} & =\\
\left(\jc_{0,1}\cdot\df^{0},\dots,\jc_{d_{1}-1,1}\cdot\df^{d_{1}-1},\jp_{1}^{\df},\dots,\jc_{0,\np}\cdot\df^{0},\dots,\jc_{d_{\np}-1,\np}\cdot\df^{d_{\np}-1},\jp_{\np}^{\df}\right)^{T}.
\end{split}
\label{eq:scalemap-def}
\end{eqnarray}
Since for small $\delta^{\left(j\right)}$ we have that $\min_{i\neq j}\left|\nd_{i}^{\df}-\nd_{j}^{\df}\right|\approx\delta^{\left(j\right)}\df$,
\eqref{eq:decimated-stability} follows from the above and \eqref{eq:square-stability}.

Experimental evidence suggests that decimation is nearly optimal in
the ``near-collision'' region, i.e.
\begin{equation}
CN_{\nd_{j}}^{\left(\df^{*}\right)}\left(\vec{x}\right)\approx CN_{\nd_{j},\nm}\left(\vec{x}\right),\qquad\nm\ensuremath{\delta^{*}}<\ensuremath{\pi}\nparams.\label{eq:empirical-optimality-of-decimation}
\end{equation}
We believe that it is an important question to provide a good quantification
of \eqref{eq:empirical-optimality-of-decimation}.

From the practical perspective, the above results suggest a nearly-optimal (in the
sense of conditioning) approach to numerically solving the system
\eqref{eq:prony} when all nodes are clustered - namely, to pick up
the $\nparams$ evenly spaced measurements
\[
\left\{ \meas[0],\meas[\df],\dots,\meas[\left(\nparams-1\right)\df]\right\} 
\]
and solve the resulting square system. 

An important feature of the decimation approach is that it introduces
\emph{aliasing} for the nodes - indeed, the system \eqref{eq:decimated-prony-def}
has $w_{j}=\jp_{j}^{\df}$ as the solution instead of $\jp_{j}$,
and therefore after solving \eqref{eq:decimated-prony-def}, the algorithm
must select the correct value for the $\df^{\text{-th}}$ root $\left(\tilde{w}_{j}\right)^{\frac{1}{\df}}$.
Thus, either the algorithm should start with an approximation of the
correct value (and thus decimation will be used as a fine-tuning technique),
or it should choose one among the $\df$ candidates via some pruning
technique - for instance, by calculating the discrepancy with the
other measurements, which were not originally utilized in the decimated
calculation.

\subsection{\label{subsec:one-point}Algebraic reconstruction}

Although many solution methods for the system \eqref{eq:prony} exist,
as we mentioned they are not well-suited for dealing with multiple
roots/eigenvalues. While averaging might work well in practice, it
is difficult to analyze rigorously, and in particular to prove the
resulting method's rate of convergence. 

In \cite{batFullFourier,batenkov_algebraic_2012} we developed a method
based on accurate solution of Prony system for resolving the Gibbs
phenomenon, i.e. for accurate recovery of a piecewise-smooth function
from its first $\nm$ Fourier coefficients. This problem arises in
spectral methods for numerical solutions of nonlinear PDEs with shock
discontinuities, and was first investigated by K.Eckhoff in the 90's
\cite{eckhoff1995arf}. The key problem was to develop a method which,
given the left-hand side $\left\{ \meas\right\} _{k=0}^{\nm-1}$ of
\eqref{eq:prony} with error decaying as $\left|\Delta\meas\right|\sim k^{-1}$
, would recover the nodes $\left\{ \nd_{j}\right\} $ with accuracy
not worse than $\left|\Delta\nd_{j}\right|\sim\nm^{-d_{j}-1}$.

Our solution was based on two main ideas:
\begin{enumerate}
\item Due to the specifics of the problem, it was sufficient to provide
a solution method as described above \emph{in the case of a single
node, i.e. }$\np=1$. 
\item The resulting system was solved by decimation, elimination of the
linear variables $\left\{ \jc_{\ell,j}\right\} $, and polynomial
root finding.
\end{enumerate}
The elimination step is a direct application of the recurrence relation
\eqref{eq:rec-rel} for the coefficients of the Prony polynomial \eqref{eq:prony-polynomial},
as follows. The (unperturbed) system \eqref{eq:prony} for $\np=1$
reads
\begin{equation}
\meas=\nd^{k}\sum_{\ell=0}^{d-1}\jc_{\ell}k^{\ell}.\label{eq:one-node-unperturbed}
\end{equation}
The corresponding decimated system \eqref{eq:decimated-prony-def}
with parameter $\df$ is
\[
n_{k}=\meas[\df k]=\left(\nd^{\df}\right)^{k}\sum_{\ell=0}^{d-1}\left(\jc_{\ell}\df^{\ell}\right)k^{\ell},\quad k=1,2,\dots,d+1.
\]

Denote $\rho\isdef\nd^{\df}$. Then clearly the sequence $\left\{ n_{k}\right\} $
satisfies $\mbox{\ensuremath{\sum_{\ell=0}^{d}n_{k+\ell}c_{\ell}}=0}$,
where the Prony polynomial is just $\mbox{\ensuremath{\left(x-\rho\right)^{d}\equiv\sum_{\ell=0}^{d}c_{\ell}x^{\ell}}}$.
That is, $\mbox{\ensuremath{c_{\ell}}=\ensuremath{\left(-1\right)^{\ell}{d  \choose \ell}\rho^{d-\ell}}}$
and we obtain that $\rho$ is one of the roots of the unperturbed
polynomial
\begin{equation}
q_{\df}\left(u\right)\isdef\sum_{\ell=0}^{d}\left(-1\right)^{\ell}{d \choose \ell}n_{\ell+1}u^{d-\ell}.\label{eq:elimination-polynomial}
\end{equation}

The algebraic reconstruction method for $\np=1$ (\cite[Algorithm 2]{batFullFourier})
is summarized in \prettyref{alg:one-point-alg} \vpageref{alg:one-point-alg}.

\begin{algorithm}
\begin{enumerate}
\item Set decimation parameter to $\df^{*}:=\left\lfloor \frac{\nm}{d+1}\right\rfloor $.
\item Construct the polynomial $\tilde{q}_{\df^{*}}\left(u\right)$ from
the given perturbed measurements $\left\{ \tilde{\meas}\right\} _{k=0}^{\nm-1}$:
\[
\tilde{q}_{\df^{*}}\left(u\right)\isdef\sum_{\ell=0}^{d}\left(-1\right)^{\ell}{d \choose \ell}\tilde{m}_{\df^{*}\left(\ell+1\right)}u^{d-\ell}.
\]
\item Set $\tilde{\rho}$ to be the root of $\tilde{q}_{\df^{*}}$ closest
to the unit circle in $\CC$.
\item Choose the solution $z^{*}$ to \eqref{eq:one-node-unperturbed} among
the $\df^{*}$ possible values of $\left(\rho^{*}\right)^{\frac{1}{\df^{*}}}$
according to available a-priori approximation.
\end{enumerate}
\caption{Algebraic reconstruction of a single node}
\label{alg:one-point-alg}

\end{algorithm}

The key result of \cite{batFullFourier} is that as $\nm\to\infty$
(and therefore $\df^{*}\to\infty$ as well), and assuming perturbation
of size $\varepsilon$ for the coefficients $\left\{ \tilde{m}_{k}\right\} $,
all the $d$ roots of $\tilde{q}_{\df^{*}}$ remain simple and well-separated,
while the corresponding perturbation of the root $\rho^{*}$ is bounded
by
\[
\left|\tilde{\rho}-\rho\right|\lessapprox\nm^{-\left(d-1\right)}\varepsilon\Longrightarrow\left|\tilde{\nd}-\nd\right|\lessapprox\nm^{-d}\varepsilon.
\]
Thus, the method is optimal - recall the condition estimate \eqref{eq:good-stability}.

The pruning step 3 was shown to be valid since the unperturbed polynomial
$q_{\df^{*}}$ has only one root on the unit circle. Regarding step
4, it was shown that a sufficiently accurate initial approximation
can be obtained by the previous method of \cite{batenkov_algebraic_2012}.
\begin{rem}
Decimation acts as a kind of regularization for the otherwise ill-conditioned
multiple root. To see why, consider the case $d=2$. Then we have
\[
\meas=\nd^{k}\left(\jc_{0}+k\jc_{1}\right).
\]
The Prony polynomial is $P\left(x\right)=\left(x-z\right)^{2}$, and
thus for each $k\in\NN$ the point $z$ is a root of
\[
q_{k}^{\#}\left(u\right)\isdef\meas u^{2}-2u\meas[k+1]+\meas[k+2].
\]
As $k\to\infty$, the above polynomial ``approaches in the limit''
\[
\frac{q_{k}^{\#}\left(u\right)}{k}\to\jc_{1}\nd^{k}\left(u-\nd\right)^{2}.
\]
Thus, a ``non-decimated'' analogue of \prettyref{alg:one-point-alg}
(such as \cite{batenkov2011accuracy,eckhoff1995arf}) would be recovering
an ``almost double'' root $\nd^{*}$, and it is well-known that
the accuracy of reconstruction in this case is only of the order $\sqrt{\varepsilon}$
when the data is perturbed by $\varepsilon$. On the other hand, $q_{\df}\left(u\right)=\meas[\df]u^{2}-2u\meas[2\df]+\meas[3\df]$,
and as $\df\to\infty$ it is easy to see that
\[
\frac{q_{\df}\left(u\right)}{\df}\to\jc_{1}\rho\left(u^{2}-4\rho u+3\rho^{2}\right)=\jc_{1}\rho\left(u-\rho\right)\left(u-3\rho\right),
\]
i.e. the limiting roots are well-conditioned.
\end{rem}

\section{\label{sec:decimated-homotopy-algorithm}Decimated homotopy algorithm}

In this section we develop the decimated homotopy algorithm, which
is a generalization of \prettyref{alg:one-point-alg} to the case
$\np>1$. We assume that the multiplicity $D$ is known, and the noise
level is small enough so that accurate recovery of the nodes by solving
the decimated system \eqref{eq:decimated-prony-def} according to
\prettyref{thm:decimated-accuracy} is possible. 

Recall that the feasible solutions are restricted to the complex torus
\[
\TT^{\np}\isdef\left\{ \vec{\nd}\in\CC^{\np}:\;\left|\left(\vec{\nd}\right)_{i}\right|=1,\;i=1,\dots,\np\right\} .
\]

\subsection{Construction of the system}

\global\long\def\symf{\tau}

Consider the decimated system \eqref{eq:decimated-prony-def} with
fixed parameter $\df$. Denote $w_{j}\isdef\nd_{j}^{\df}$. The decimated
measurements $\left\{ n_{k}\right\} _{k=0}^{\nparams-1}$ satisfy
for each $k=0,\dots,\np-1$
\[
\sum_{i=0}^{d}n_{k+i}c_{i}=0,
\]
where $c_{i}$ are the coefficients of the Prony polynomial
\[
P\left(x\right)=\prod_{j=1}^{\np}\left(x-w_{j}\right)^{d_{j}}\equiv\sum_{\ell=0}^{d}c_{\ell}x^{\ell}.
\]
Let $\sigma_{i}\left(x_{1},\dots,x_{d}\right)$ denote the elementary
symmetric polynomial of order $i$ in $d$ variables. Then we have
\begin{equation}
c_{\ell}=\left(-1\right)^{d-\ell}\sigma_{d-\ell}\left(\underbrace{w_{1},\dots,w_{1}}_{\times d_{1}},\dots,\underbrace{w_{\np},\dots,w_{\np}}_{\times d_{\np}}\right)\isdef\symf_{\ell}\left(w_{1},\dots,w_{\np}\right).\label{eq:symf-def}
\end{equation}

Thus the point $\vec{w}=\left(w_{1},\dots,w_{\np}\right)\in\TT^{\np}$
is a zero of the $\np\times\np$ polynomial system
\begin{equation}
\left\{ f_{k}^{\left(\df\right)}\left(\vec{u}\right)\isdef\sum_{i=0}^{d}n_{k+i}\symf_{i}\left(\vec{u}\right)=0\right\} _{k=0,\dots,\np-1}.\label{eq:the-hankel-system}
\end{equation}
This Hankel-type system is therefore our proposed generalization to
the polynomial equation \eqref{eq:elimination-polynomial}. 
\begin{example}
$\np=2,\;d_{j}=2$. The system \eqref{eq:the-hankel-system} reads{\footnotesize{}
\[
\begin{bmatrix}f_{0}\left(\vec{u}\right)\\
f_{1}\left(\vec{u}\right)
\end{bmatrix}=\left[\begin{matrix}n_{0}u_{1}^{2}u_{2}^{2}+n_{1}\left(-2u_{1}^{2}u_{2}-2u_{1}u_{2}^{2}\right)+n_{2}\left(u_{1}^{2}+4u_{1}u_{2}+u_{2}^{2}\right)+n_{3}\left(-2u_{1}-2u_{2}\right)+n_{4}\\
n_{1}u_{1}^{2}u_{2}^{2}+n_{2}\left(-2u_{1}^{2}u_{2}-2u_{1}u_{2}^{2}\right)+n_{3}\left(u_{1}^{2}+4u_{1}u_{2}+u_{2}^{2}\right)+n_{4}\left(-2u_{1}-2u_{2}\right)+n_{5}
\end{matrix}\right]
\]
}{\footnotesize \par}
\end{example}

\subsection{Recovering the solution}

Generalizing the root finding step of \prettyref{alg:one-point-alg},
we propose to use the homotopy continuation method in order to find
all the isolated solutions of the (perturbed) system \eqref{eq:the-hankel-system}. 

While a-priori it is not clear whether the variety defined by \eqref{eq:the-hankel-system}
has positive-dimensional components, we show in \prettyref{sec:Analysis}
below that our wanted solution is indeed isolated, and therefore the
homotopy will find it. Furthermore, by analyzing the Jacobian of the
polynomial map at the solution and comparing it with the estimate
\eqref{eq:decimated-stability}, we show that the obtained accuracy
is optimal.

We now consider the question of how to recover the correct solution
of the original problem \eqref{eq:prony} from among all the isolated
solutions
\[
{\cal S}=\left\{ \vec{u}_{1},\dots,\vec{u}_{S}\right\} 
\]
of \eqref{eq:the-hankel-system}. Two issues need to be addressed.
\begin{enumerate}
\item \emph{Spurious solutions.} Transformation to the Hankel-type polynomial
system introduces spurious solutions $\vec{u}\in\CC^{\np}$ which
are not in $\TT^{\np}$ and therefore cannot be equal to the $p$-th power
of a solution to the original system \eqref{eq:prony}.
\item \emph{Aliasing}. Given a solution $\vec{u}\in{\cal S}\cap\TT^{\np}$
to the $\df$-decimated system \eqref{eq:the-hankel-system}, there
are in general $\df^{\np}$ possible corresponding solutions to \eqref{eq:prony},
namely
\[
{\cal Z}_{\df}\left(\vec{u}\right)\isdef\left\{ \left(\nd_{1},\dots,\nd_{\np}\right)\in\TT^{\np}:\;\left(\nd_{i}\right)^{\df}=\left(\vec{u}\right)_{i}\right\} .
\]
\end{enumerate}
Overall, the set of all possible candidate solutions is
\[
{\cal G}:=\bigcup_{\vec{u}\in{\cal S}}{\cal Z}_{\df}\left(\vec{u}\right).
\]

We suggest several pruning strategies.

\paragraph{Exhaustive search}

Since we have access to the original system, we can select the solution
giving the smallest residual for the non-decimated equations \eqref{eq:the-hankel-system},
i.e.
\begin{equation}
\vec{\nd}^{*}\isdef\argmin_{\vec{\nd}\in{\cal G}}\sum_{k=0}^{d-1}\left|f_{k}^{\left(1\right)}\left(\vec{z}\right)\right|.\label{eq:exhaustive}
\end{equation}
Optionally, the upper index in the summation in \eqref{eq:exhaustive}
can be increased to $\nm-d-1$.

\paragraph{Pre-filtering}

Instead of considering the whole set ${\cal G}$, one can first prune
${\cal S}$ and work with the solution which has smallest distance
to $\TT^{\np}$, i.e.:
\begin{eqnarray}
\begin{split}\vec{u}^{*} & \leftarrow\argmin_{\vec{u}_{k}\in{\cal S}}\max_{i=1,\dots,\np}\left|1-\left|\left(\vec{u}_{k}\right)_{i}\right|\right|\\
\left(\vec{u}^{*}\right)_{i} & \leftarrow\frac{\left(\vec{u}^{*}\right)_{i}}{\left|\left(\vec{u}^{*}\right)_{i}\right|}\\
{\cal G} & \leftarrow{\cal Z}_{\df}\left(\vec{u}^{*}\right).
\end{split}
\label{eq:heuristic-prefilter}
\end{eqnarray}

\paragraph{Using an initial approximation}

In some applications it may be possible to obtain an a-priori approximation
to the location of the desired solution $\vec{\nd}$. So suppose that
the algorithm is provided with $\vec{\nd}_{init}\in\TT^{\np}$ and
a threshold $\eta>0$, then additional pruning can be achieved by
setting
\begin{equation}
{\cal G}\leftarrow\left\{ \vec{\nd}\in{\cal G}:\;\left|\vec{\nd}-\vec{\nd}_{init}\right|\leqslant\eta\right\} .\label{eq:heuristic-approx}
\end{equation}

The strategies can be mixed, i.e. pre-filtering \eqref{eq:heuristic-prefilter}
can be followed by either \eqref{eq:exhaustive} or \eqref{eq:heuristic-approx}
etc.

The decimated homotopy is summarized in \prettyref{alg:decimated-homotopy}\vpageref{alg:decimated-homotopy}.

In the next section we prove that for small enough noise, the exhaustive
search \eqref{eq:exhaustive} is guaranteed to produce the approximation
to the original solution, which is near-optimal. However, this strategy
quickly becomes prohibitive. The pruning strategy \eqref{eq:heuristic-prefilter}
combined with \eqref{eq:heuristic-approx} is empirically shown to
be as accurate and much faster, see \prettyref{sec:numerical-experiments}.

\begin{algorithm}
Given: $\left(\mm_{0},\dots,\mm_{\nm-1}\right)\in\CC^{\nm}$, multiplicity
$D$.
\begin{enumerate}
\item Set decimation parameter $\df^{*}=\left\lfloor \frac{\nm}{\nparams}\right\rfloor $.
\item Construct the system
\[
{\cal H}_{\df^{*}}:\;\left\{ f_{k}^{\left(\df^{*}\right)}\left(\vec{u}\right)\isdef\sum_{i=0}^{d}\mm_{\df^{*}\left(k+i\right)}\symf_{i}\left(\vec{u}\right)=0\right\} _{k=0,\dots,\np-1}.
\]
\item Solve ${\cal H}_{\df^{*}}$ by homotopy continuation method. Let ${\cal S}$
be the collection of its isolated solutions.
\item Select $\vec{\nd}^{*}\in{\cal G}$ using either \eqref{eq:exhaustive}
, \eqref{eq:heuristic-prefilter}, \eqref{eq:heuristic-approx}, or
a combination thereof.
\end{enumerate}
\caption{Decimated homotopy algorithm}
\label{alg:decimated-homotopy}

\end{algorithm}

\section{\label{sec:Analysis}Analysis}

In this section we analyze the proposed algorithm in the case of a
single cluster. We show that for small enough noise level, the exhaustive
search \eqref{eq:exhaustive} is guaranteed to produce a near-optimal
approximation to the original solution.

\subsection{Statement of the results}

\global\long\def\jacH{{\cal D}}

Let $\vec{x}\in\ddd$ be a data point \eqref{eq:data-point} satisfying
the conditions of \prettyref{thm:decimated-accuracy}, and $\df^{*}$
be the corresponding decimation parameter. The corresponding decimated
measurements $n_{k}=n_{k}\left(\vec{x}\right)$ are defined as in
\eqref{eq:decimated-prony-def}:
\[
n_{k}=\sum_{j=1}^{\np}w_{j}^{k}\sum_{\ell=0}^{d_{j}-1}b_{\ell,j}k^{\ell},
\]
where $w_{j}=\nd_{j}^{\df^{*}}$ and $b_{\ell,j}=\jc_{\ell,j}\df^{*\ell}$.
As before, $\vec{w}=\vec{w}\left(\vec{x}\right)=\left(w_{1},\dots,w_{\np}\right)$.
\begin{defn}
Let $\jacH_{\vec{x}}\left(\vec{t}\right)\in\CC^{\np\times\np}$ denote
the Jacobian matrix of the system \eqref{eq:the-hankel-system} at
the point $\vec{u}=\vec{t}$:
\[
\jacH_{\vec{x}}\left(\vec{t}\right)\isdef\left(\frac{\partial f_{k}\left(\vec{x}\right)}{\partial u_{j}}\Big\vert_{\vec{t}}\right)_{j=1,\dots,\np}^{k=0,\dots,\np-1}.
\]
\end{defn}
\begin{thm}
\label{thm:hankelized-well-posedness}Let the conditions of \prettyref{thm:decimated-accuracy}
be satisfied. Then $\det\jacH_{\vec{x}}\left(w\right)\neq0$. Therefore,
$\vec{u}=\vec{w}$ is an isolated solution of \eqref{eq:the-hankel-system}.
\end{thm}
We adopt the definitions from \cite{stetter_numerical_2004} (in particular,
see Section 3.2.3 and Section 9.1.2) for measuring linearized sensitivity
of the solutions of empirical polynomial systems with respect to perturbations
of the coefficients. Not surprisingly, they parallel our earlier definitions
of conditioning (see \prettyref{sec:prior-work}).

Let $j$ be a multi-index, and denote by $\vec{u}^{j}$ the monomial
$u_{1}^{j_{1}}\cdots u_{\np}^{j_{\np}}$. For $k=0,\dots,\np-1$,
let $\alpha_{k,j}$ denote the coefficient of $\vec{u}^{j}$ in the
equation number $k$ in \eqref{eq:the-hankel-system}. Finally, let
$J_{k}$ denote the set of multi-indices $j$ for which $\alpha_{k,j}\neq0$.
\begin{defn}
\label{def:system-sensitivity}Let $\vec{u}$ be a zero of the system
\eqref{eq:the-hankel-system}. Assume that the coefficient $\alpha_{k,j}$
is perturbed by at most $\Delta\alpha_{k,j}$. The linearized sensitivity
of the $i$-th component of $\vec{u}$ to the change in the coefficients
of the system is
\[
\kappa_{i}\left(\vec{u}\right):=\sum_{k=0}^{\np-1}\left|K_{i,k}\right|\left(\sum_{j\in J_{k}}\left|\vec{u}^{j}\right|\right)\max_{j}\left|\Delta\alpha_{k,j}\right|,
\]
where $K_{i,k}$ are the entries of the matrix $\left\{ \jacH_{\vec{x}}\left(\vec{u}\right)\right\} ^{-1}$.
\end{defn}
\begin{thm}
\label{thm:accuracy}Let the conditions of \prettyref{thm:decimated-accuracy}
be satisfied. For small enough noise $\er\ll1$ in the right-hand
side of \eqref{eq:prony}, \prettyref{alg:decimated-homotopy} with
\eqref{eq:exhaustive} recovers the solution $\vec{u}=\vec{w}$ of
\eqref{eq:the-hankel-system} with accuracy
\begin{equation}
\kappa_{i}\left(\vec{w}\right)\leqslant C^{\left(4\right)}\frac{1}{\left|\jc_{d_{i}-1,i}\right|\nm^{\nparams-1}}\left(\frac{1}{\delta^{\left(i\right)}}\right)^{\nparams-d_{i}}\er.\label{eq:acc-decim-hankel}
\end{equation}
Consequently, the original solution $\vec{\nd}$ of \eqref{eq:prony}
is recovered with accuracy
\begin{equation}
\left|\Delta\nd_{i}\right|\leqslant C^{\left(5\right)}\frac{1}{\left|\jc_{d_{i}-1,i}\right|\nm^{\nparams}}\left(\frac{1}{\delta^{\left(i\right)}}\right)^{\nparams-d_{i}}\er.\label{eq:acc-final}
\end{equation}
Here $C^{\left(4\right)}$ and $C^{\left(5\right)}$ are constants
depending only on the multiplicity vector $D$.
\end{thm}
Comparing the bounds \eqref{eq:acc-final} and \eqref{eq:decimated-stability},
we conclude that the proposed algorithm is optimal, up to a constant,
in the considered setting.

\subsection{\label{subsec:proofs}Proofs}

\global\long\def\opp{{\cal E}}
\global\long\def\dgmt{{\cal B}}

We start by deriving explicit expressions for the entries of $\jacH$.
\begin{lem}
\label{lem:part-der}For $j=1,\dots,\np$ and arbitrary $k\in\NN$
we have
\[
\frac{\partial f_{k}}{\partial u_{j}}\Big\vert_{\vec{w}}=-d_{j}!w_{j}^{k+d_{j}-1}b_{d_{j}-1,j}\prod_{i\neq j}\left(w_{j}-w_{i}\right)^{d_{i}}.
\]
\end{lem}
\begin{proof}
Considering the coefficients $n_{k}$ as functions of $\vec{w}$,
we have the identity
\[
f_{k}\left(\left\{ n_{k}\left(\vec{w}\right)\right\} ,\vec{w}\right)=\sum_{i=0}^{d}n_{k+i}\left(\vec{w}\right)\symf_{i}\left(\vec{w}\right)\equiv0.
\]
Thus, for each $w_{j}$ the total derivative $\frac{df_{k}}{dw_{j}}\left(\left\{ n_{k}\right\} ,\vec{u}\right)$
vanishes on $\vec{u}\left(\vec{w}\right)=\vec{w}$. By the chain rule
\begin{eqnarray*}
\frac{d}{dw_{j}}f_{k}\left(\left\{ n_{k}\right\} ,\vec{u}\right) & = & \sum_{i=0}^{d}\frac{\partial f_{k}}{\partial n_{k+i}}\frac{\partial n_{k+i}}{\partial w_{j}}+\sum_{\ell=1}^{\np}\frac{\partial f_{k}}{\partial u_{\ell}}\frac{\partial u_{\ell}}{\partial w_{j}}\\
 & = & \sum_{i=0}^{d}\symf_{i}\left(\vec{w}\right)\left(k+i\right)w_{j}^{k+i-1}\sum_{\ell=0}^{d_{j}-1}b_{\ell,j}\left(k+i\right)^{\ell}+\frac{\partial f_{k}}{\partial u_{j}}\left(\vec{w}\right)\\
 & = & 0.
\end{eqnarray*}
Let $r_{j}\left(k\right)$ denote the following polynomial in $k$
of degree $d_{j}$:
\[
r_{j}\left(k\right)\isdef\sum_{\ell=0}^{d_{j}-1}b_{\ell,j}k^{\ell+1}.
\]
Then we have
\begin{equation}
\frac{\partial f_{k}}{\partial u_{j}}\left(\vec{w}\right)=-w_{j}^{k-1}\sum_{i=0}^{d}w_{j}^{i}\symf_{i}\left(\vec{w}\right)r_{j}\left(k+i\right).\label{eq:der-as-discr-seq}
\end{equation}

We now employ standard tools from finite difference calculus \cite{elaydi2005ide}.
Consider the right-hand side of \eqref{eq:der-as-discr-seq} as a
discrete sequence depending on a running index $k$. Let $\shift=\shift_{k}$
denote the discrete shift operator in $k$, i.e. for any discrete
sequence $g\left(k\right)$ we have 
\[
\shift g\left(k\right)=\left(\shift g\right)\left(k\right)=g\left(k+1\right).
\]
Let us further denote by $\dl\isdef\shift-\id$ the discrete differentiation
operator ($\id$ is the identity operator). Now consider the difference
operator
\begin{equation}
\opp_{j}\isdef\prod_{i=1}^{\np}\left(w_{j}\shift-w_{i}\id\right)^{d_{i}}.\label{eq:dfop}
\end{equation}
Recall the definition of $\symf_{i}$ from \eqref{eq:symf-def}. Opening
parenthesis, we obtain that for any $g\left(k\right)$
\[
\opp_{j}g\left(k\right)=\sum_{i=0}^{d}w_{j}^{i}\symf_{i}\left(\vec{w}\right)g\left(k+i\right).
\]
Therefore
\begin{eqnarray*}
\frac{\partial f_{k}}{\partial u_{j}}\left(\vec{w}\right) & = & -w_{j}^{k-1}\opp_{j}r_{j}\left(k\right).
\end{eqnarray*}
Since the linear factors in \eqref{eq:dfop} commute, we proceed as
follows:
\begin{equation}
-w_{j}^{k-1}\opp_{j}r_{j}\left(k\right)=-w_{j}^{k-1}\prod_{i\neq j}\left(w_{j}\shift-w_{i}\id\right)^{d_{i}}\left(w_{j}\dl\right){}^{d_{j}}r_{j}\left(k\right).\label{eq:temp1}
\end{equation}
It is an easy fact (e.g. \cite{elaydi2005ide}) that for any polynomial
$p\left(k\right)$ of degree $n$ and leading coefficient $a_{0}$,
we have that
\[
\dl^{n}p\left(k\right)=a_{0}n!.
\]
Since $r_{j}\left(k\right)$ has degree $d_{j}$, we obtain that
\begin{equation}
\dl^{d_{j}}r_{j}\left(k\right)=d_{j}!b_{d_{j}-1,j}.\label{eq:dl-pol}
\end{equation}
Furthermore, applying the operator $w_{j}\shift-w_{i}\id$ to a constant
sequence $\mbox{c\ensuremath{\left(k\right)}=c}$ gives
\begin{equation}
\left(w_{j}\shift-w_{i}\id\right)c=\left(w_{j}-w_{i}\right)c.\label{eq:linear-on-const}
\end{equation}
Plugging \eqref{eq:dl-pol} and \eqref{eq:linear-on-const} into \eqref{eq:temp1}
we get:
\[
\frac{\partial f_{k}}{\partial u_{j}}\left(\vec{w}\right)=-w_{j}^{k+d_{j}-1}d_{j}!b_{d_{j}-1,j}\prod_{i\neq j}\left(w_{j}-w_{i}\right)^{d_{i}},
\]
completing the proof of \prettyref{lem:part-der}.
\end{proof}

\begin{example}
For $\np=3,\;d_{j}=2$ we have{\footnotesize{}
\[
\jacH=\left[\begin{matrix}-2b_{11}w_{1}\left(w_{1}-w_{2}\right)^{2}\left(w_{1}-w_{3}\right)^{2} & -2b_{12}w_{2}\left(w_{1}-w_{2}\right)^{2}\left(w_{2}-w_{3}\right)^{2} & -2b_{13}w_{3}\left(w_{1}-w_{3}\right)^{2}\left(w_{2}-w_{3}\right)^{2}\\
-2b_{11}w_{1}^{2}\left(w_{1}-w_{2}\right)^{2}\left(w_{1}-w_{3}\right)^{2} & -2b_{12}w_{2}^{2}\left(w_{1}-w_{2}\right)^{2}\left(w_{2}-w_{3}\right)^{2} & -2b_{13}w_{3}^{2}\left(w_{1}-w_{3}\right)^{2}\left(w_{2}-w_{3}\right)^{2}\\
-2b_{11}w_{1}^{3}\left(w_{1}-w_{2}\right)^{2}\left(w_{1}-w_{3}\right)^{2} & -2b_{12}w_{2}^{3}\left(w_{1}-w_{2}\right)^{2}\left(w_{2}-w_{3}\right)^{2} & -2b_{13}w_{3}^{3}\left(w_{1}-w_{3}\right)^{2}\left(w_{2}-w_{3}\right)^{2}
\end{matrix}\right]
\]
}{\footnotesize \par}
\end{example}
\begin{defn}
Let $V\left(\vec{w}\right)$ denote the $\np\times\np$ Vandermonde
matrix on the nodes $\left\{ w_{1},\dots,w_{\np}\right\} .$ For example,
if $\np=3$ we have:
\[
V\left(\vec{w}\right)=\left[\begin{matrix}1 & 1 & 1\\
w_{1} & w_{2} & w_{3}\\
w_{1}^{2} & w_{2}^{2} & w_{3}^{2}
\end{matrix}\right].
\]
\end{defn}
\begin{cor}
Let $\vec{y}=\scmap[\df^{*}]\left(\vec{x}\right)$ where $\scmap[\df^{*}]$
is the scaling mapping \eqref{eq:scalemap-def}. Denote by $\dgmt\left(\vec{y}\right)$
the following $\np\times\np$ diagonal matrix:
\[
\dgmt\left(\vec{y}\right)\isdef\diag_{j=1,\dots,\np}\left\{ -d_{j}!b_{d_{j}-1,j}\prod_{i\neq j}\left(w_{j}-w_{i}\right)^{d_{i}}\right\} .
\]
Then we have the factorization
\begin{equation}
\jacH_{\vec{x}}\left(\vec{w}\right)=V\left(\vec{w}\right)\dgmt\left(\vec{y}\right).\label{eq:jac-fact}
\end{equation}
\end{cor}
\begin{proof}
Directly from \prettyref{lem:part-der}.
\end{proof}
\begin{minipage}[t]{1\columnwidth}%
\end{minipage}
\begin{proof}[Proof of \prettyref{thm:hankelized-well-posedness}]
According to our assumptions, both $V\left(\vec{w}\right)$ and $\dgmt\left(\vec{y}\right)$
are non-singular (in particular, since $\jc_{d_{j}-1,j}\neq0$ and
$\min_{i\neq j}\left|w_{i}-w_{j}\right|>0$). Using \eqref{eq:jac-fact}
we conclude that $\jacH$ is invertible. The conclusion that $\vec{u}=\vec{w}$
is isolated is a standard fact about multivariate nonlinear systems,
see e.g. \cite{dayton_multiple_2011}.
\end{proof}
\begin{minipage}[t]{1\columnwidth}%
\end{minipage}
\begin{proof}[Proof of \prettyref{thm:accuracy}]
From \eqref{eq:jac-fact} we have
\[
\left\{ \jacH_{\vec{x}}\left(\vec{w}\right)\right\} ^{-1}=\diag_{i=1,\dots,\np}\left\{ -\frac{1}{d_{i}!b_{d_{i}-1,i}}\prod_{\ell\neq i}\left(w_{\ell}-w_{i}\right)^{-d_{i}}\right\} V\left(\vec{w}\right)^{-1}.
\]
Let $V\left(\vec{w}\right)^{-1}=\left(v_{\alpha,\beta}\right)_{\alpha,\beta=1,\dots,\np}$.
Using the classical estimates by Gautschi \cite{gautschi_inverses_1962},
we have that 
\[
\sum_{k=0}^{\np-1}\left|v_{i,k}\right|\leqslant\prod_{j\neq i}\frac{1+\left|w_{j}\right|}{\left|w_{j}-w_{i}\right|}=2^{\np-1}\prod_{j\neq i}\left|w_{j}-w_{i}\right|^{-1}.
\]
Therefore
\begin{eqnarray*}
\sum_{k=0}^{\np-1}\left|K_{i,k}\right| & = & \frac{1}{d_{i}!b_{d_{i}-1,i}}\prod_{\ell\neq i}\left(w_{\ell}-w_{i}\right)^{-d_{i}}\sum_{k=0}^{\np-1}\left|v_{i,k}\right|\\
 & \leqslant & \frac{2^{\np-1}}{d_{i}!\left|\jc_{d_{i}-1,i}\right|\df^{*d_{i}-1}}\prod_{\ell\neq i}\left(w_{\ell}-w_{i}\right)^{-d_{i}-1}\\
 & \leqslant & \frac{2^{\np-1}}{d_{i}!\left|\jc_{d_{i}-1,i}\right|\df^{*d_{i}-1}}\left(\frac{1}{\df^{*}\delta^{\left(i\right)}}\right)^{\nparams-d_{i}-1}\\
 & = & \frac{2^{\np-1}}{d_{i}!\left|\jc_{d_{i}-1,i}\right|}\frac{1}{\df^{*\nparams-2}}\left(\frac{1}{\delta^{\left(i\right)}}\right)^{\nparams-d_{i}-1}.
\end{eqnarray*}
Now clearly there exists a constant $C_{1}=C_{1}\left(\np\right)$
for which
\[
\left(\sum_{j\in J_{k}}\left|\vec{w}^{j}\right|\right)\max_{j}\left|\Delta\alpha_{k,j}\right|\leqslant C_{1}\er.
\]
Since $\nm\delta^{\left(i\right)}<\pi\nparams$ and $\df^{*}=\frac{\nm}{\nparams}$,
we have the bound 
\begin{eqnarray*}
\kappa_{i}\left(\vec{w}\right) & \leqslant & C_{2}\frac{1}{\left|\jc_{d_{i}-1,i}\right|\nm^{\nparams-2}}\left(\frac{1}{\delta^{\left(i\right)}}\right)^{\nparams-d_{i}-1}\er\\
 & \leqslant & C_{3}\frac{1}{\left|\jc_{d_{i}-1,i}\right|\nm^{\nparams-1}}\left(\frac{1}{\delta^{\left(i\right)}}\right)^{\nparams-d_{i}}\er,
\end{eqnarray*}
proving \eqref{eq:acc-decim-hankel} with $C^{\left(4\right)}=C_{3}$. 

Because extraction of $\df^{*}$-th root reduces error by a factor
of $\df^{*}$, this immediately implies \eqref{eq:acc-final} with
$C^{\left(5\right)}=\nparams C^{\left(4\right)}$.

Since the homotopy algorithm converges to the exact solution of the
approximate system \eqref{eq:the-hankel-system}, and since $\er$
is assumed to be sufficiently small, the exhaustive search \eqref{eq:exhaustive}
must produce the exact solution to the perturbed \prettyref{prob:prony}
\textendash{} otherwise there would be two non-proportional vectors
in the nullspace of the Hankel matrix $H_{\dg}$ \eqref{eq:hankel-data-matrix},
and this is impossible since the rank of $H_{\dg}$ is known to be
exactly $\dg$.
\end{proof}

\section{\label{sec:numerical-experiments}Numerical experiments}

\subsection{Setup}

We chose the model \eqref{eq:prony} with two closely spaced nodes,
varying multiplicity and random linear coefficients $\left\{ \jc_{\ell,j}\right\} $
.

We have implemented two pruning variants:
\begin{enumerate}
\item the exhaustive search \eqref{eq:exhaustive};
\item combination of \eqref{eq:heuristic-prefilter} and \eqref{eq:heuristic-approx}
(referred to as \emph{filtering} below).
\end{enumerate}
Choosing the overall number of measurements to be relatively high
(1000-4000), we varied the decimation parameter $\df$ and compared
the reconstruction error for \prettyref{alg:decimated-homotopy} with
filtering above (referred to as \emph{DH} below), and the generalized
ESPRIT algorithm \cite{badeau2006high,badeau2008performance,stoica2005spectral}
(see also \cite{batenkov2011accuracy}), one of the best performing
subspace methods for estimating parameters of the Prony systems \eqref{eq:prony}
with white Gaussian noise $\epsilon_{k}.$ The noise level in our
experiments was relatively small. 

In addition to the reconstruction error, for each run we also computed
both the full and decimated condition numbers $CN_{\nd_{j},\nm}$
and $CN_{\nd_{j}}^{\left(\df\right)}$ from their respective definitions
\eqref{eq:cn-full-def} and \eqref{eq:cn-dec-def}.

Additional implementation details:
\begin{enumerate}
\item PHCPACK \cite{verschelde_polynomial_2011} Release 2.3.96 was used
as the homotopy continuation solver. It was called via its MATLAB
interface PHCLab \cite{guan_phclab:_2008}. All tests were run on
Apple MacBook Pro 2.4 GHz Intel Core i5 with 8GB RAM under OSX 10.10.4.
\item We used the value $\eta\approx\nm^{-1}$ for the heuristic \eqref{eq:heuristic-approx}.
\item The node selection in generalized ESPRIT was done via $k$-means clustering
on the output of the eigenvalue step.
\end{enumerate}

\subsection{Results}

The results of experiments are presented in \prettyref{tbl:exh-prun-timings},
\prettyref{tbl:dh-esprit-timing} and \prettyref{fig:dh-esprit}.
They can be summarized as follows:
\begin{enumerate}
\item The exhaustive search is accurate, but time-prohibitive even for moderate
values of $\df$ (the number of solutions considered is on the order
of $\np!d^{\np}\df^{\np}$).
\item The accuracy of DH surpasses ESPRIT by several significant digits
in the near-collision region $\nm\delta\ll1$.
\item DH achieves desired accuracy in larger number of cases.
\end{enumerate}
\begin{figure}
\subfloat[]{\includegraphics[clip,width=0.35\paperwidth]{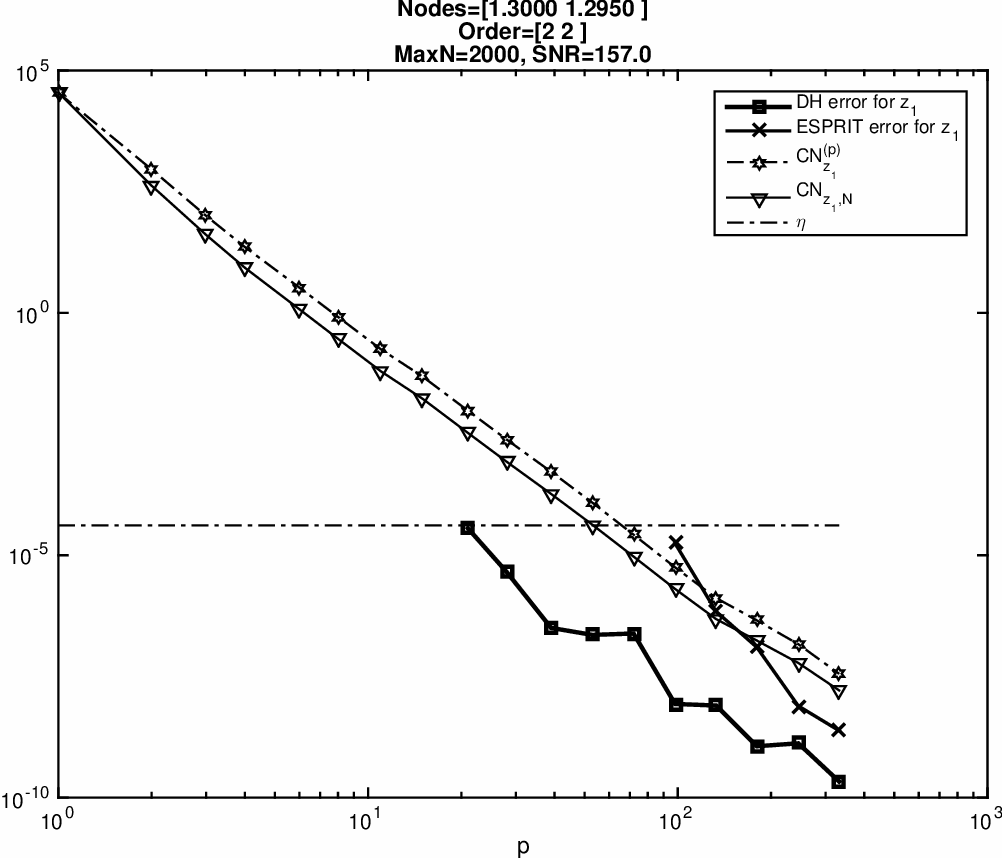}

}\subfloat[]{\includegraphics[width=0.35\paperwidth]{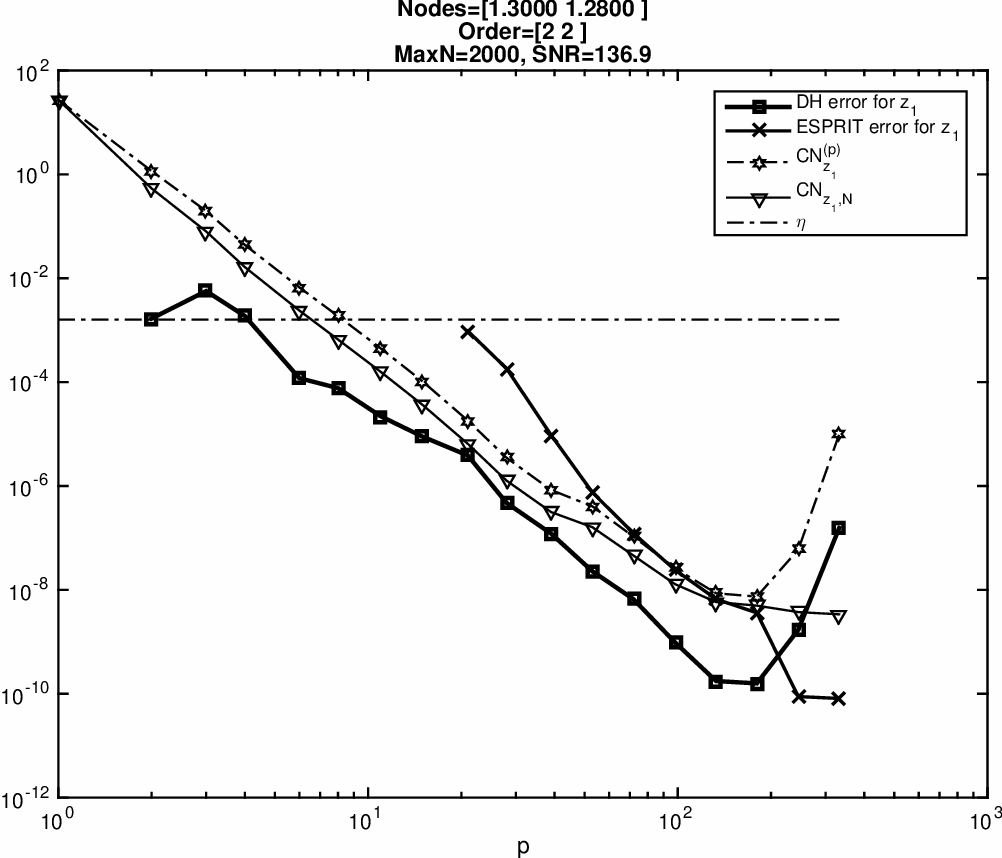}

\label{subfig:exp2}}

\subfloat[]{\includegraphics[width=0.35\paperwidth]{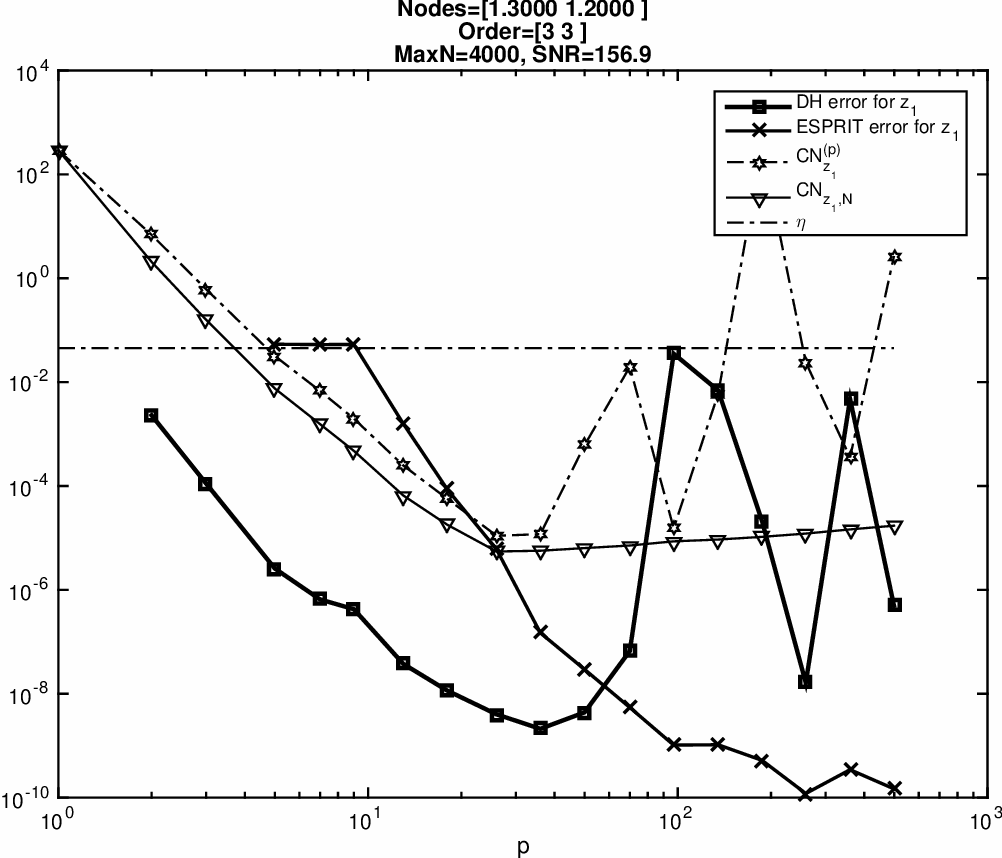}}\subfloat[]{\includegraphics[width=0.35\paperwidth]{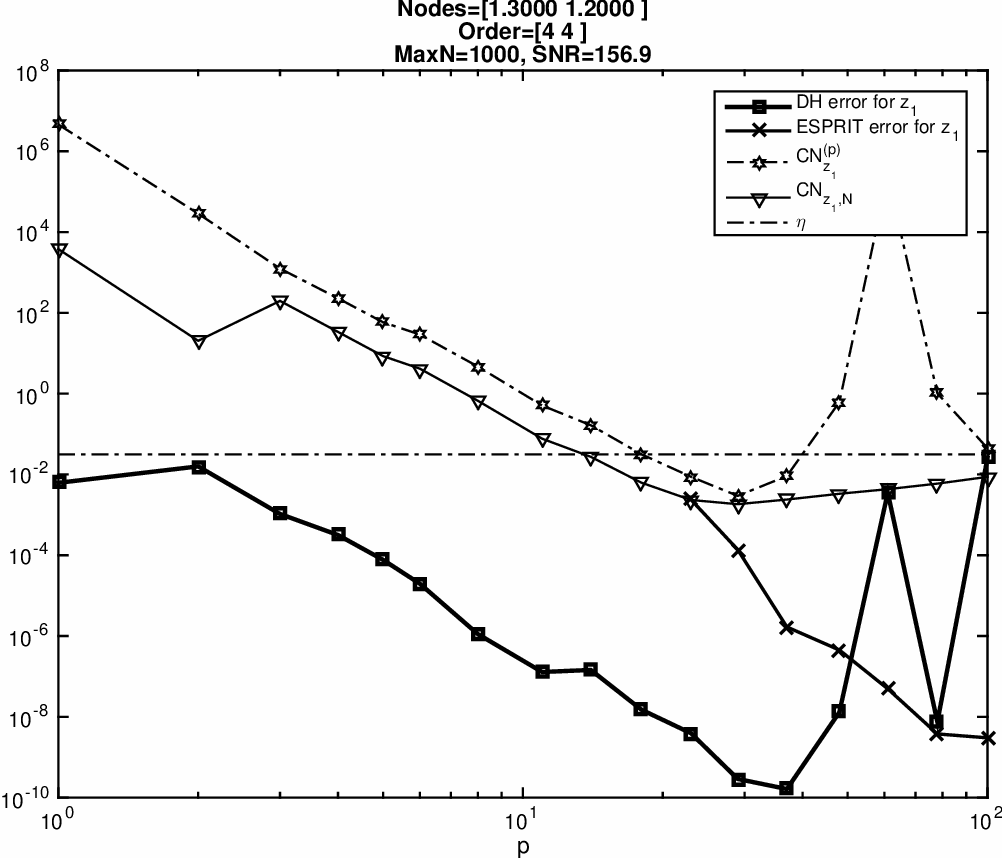}}

\caption{DH vs. ESPRIT. Also plotted are condition numbers, both full and decimated,
as well as the threshold $\eta$ used in the experiments.}
\label{fig:dh-esprit}
\end{figure}

\begin{table}
\begin{centering}
\subfloat[Exhaustive search. Pruning step only.]{\begin{centering}
\begin{tabular}{|c|c|c|}
\hline 
$\df$ & time (sec) & rec.error\tabularnewline
\hline 
\hline 
1 & 1 & 0.005\tabularnewline
\hline 
3 & 8 & 0.001\tabularnewline
\hline 
5 & 25 & 0.0002\tabularnewline
\hline 
10 & 99 & 0.00009\tabularnewline
\hline 
\end{tabular}
\par\end{centering}
\label{tbl:exhaustive-time}} \subfloat[Filtering. Overall timings]{\begin{centering}
\begin{tabular}{|c|c|c|}
\hline 
$\df$ & time (sec) & rec.error\tabularnewline
\hline 
\hline 
3 & 0.9 & 0.0002\tabularnewline
\hline 
4 & 1.1 & 0.0001\tabularnewline
\hline 
6 & 1.1 & 0.0004\tabularnewline
\hline 
8 & 1.0 & 0.00008\tabularnewline
\hline 
\end{tabular}
\par\end{centering}
\label{tbl:pruning-timing}}
\par\end{centering}
\caption{Running times for exhaustive search and filtering. The experimental
setup is the same as in Figure \ref{subfig:exp2}. Each table describes
a separate experiment, therefore the reconstruction errors are slightly
different.}

\label{tbl:exh-prun-timings}
\end{table}

\begin{table}
\begin{centering}
\begin{tabular}{|c|c|c|c|c|}
\hline 
$p$ & ESPRIT & DH\_Create & DH\_Run & DH\_Select\tabularnewline
\hline 
\hline 
120 & 0.13 & 0.84 & 0.11 & 0.001\tabularnewline
\hline 
169 & 0.21 & 0.75 & 0.09 & 0.001 \tabularnewline
\hline 
238 & 0.21 & 0.74 & 0.10 & 0.002\tabularnewline
\hline 
335 & 0.34 & 0.89 & 0.10 & 0.003\tabularnewline
\hline 
\end{tabular}
\par\end{centering}
\caption{DH vs ESPRIT timings (sec). The columns for DH correspond to the system
construction, PHC run and the pruning steps.}

\label{tbl:dh-esprit-timing}
\end{table}

Some additional remarks:
\begin{enumerate}
\item The number of solutions of the system \eqref{eq:the-hankel-system}
was equal to $\np!d^{\np}$ ($\np$=number of nodes, $d=d_{j}$=degree). 
\item Running times are better for DH when $\df$ is large, because the
selection step of \prettyref{alg:decimated-homotopy} is $O\left(\nm\right)$,
while the cost of full SVD for ESPRIT is $O\left(\nm^{2}\nparams\right)$.
See in particular \prettyref{tbl:dh-esprit-timing}. We did not observe
any unusual memory consumption during the execution of the algorithms.
\item Condition number estimates are somewhat pessimistic, nevertheless
indicating the order of error decay in a relatively accurate fashion.
The periodic pattern is well-predicted by the theory, see \cite{batenkov_numerical_2014}.
\end{enumerate}

\section{\label{sec:future-work}Discussion}

In this paper we presented a novel algorithm, Decimated Homotopy,
for numerical solution of systems of Prony type \eqref{eq:prony}
with nodes on the unit circle, $\left|\nd_{j}\right|=1$ which are
closely spaced. Analysis shows that the produced solutions have near-optimal
numerical accuracy. Numerical experiments demonstrate that the pruning
heuristics are efficient in practice, and the algorithm provides reconstruction
accuracy several orders of magnitude better than the standard ESPRIT
algorithm. The pruning\eqref{eq:heuristic-prefilter} will be justified
in the case that the system ${\cal H}_{\df^{*}}$ has no spurious
solutions on the torus $\TT^{\np}$. This seems to hold in practice,
and therefore a theoretical analysis of this question would be desirable.
On the other hand, initial approximations of order $\eta\approx\nm^{-1}$
can presumably be obtained by existing methods with lower-order multiplicity
(similar to what was done in \cite{batenkov_algebraic_2012}).

Another important question of interest is robust detection of these
near-singular situations, i.e. correct identification of the collision
pattern $D$. While the integer $\dg$ can be estimated via numerical
rank computation of the Hankel matrix $H_{d}$ \eqref{eq:hankel-data-matrix}
(see e.g. \cite{cadzow_total_1994} and also a randomized approach
\cite{kaltofen_fast_2012}), the determination of the individual components
of $D$ is a more delicate task, which requires an accurate estimation
of the distance from the data point to the nearest ``pejorative''
manifold of larger multiplicity, and comparing it with the a-priori
bound $\acc$ on the error. We hope that the present (and future)
symbolic-numeric techniques such as \cite{dayton_multiple_2011,pope_nearest_2009},
combined with description of singularities of the Prony mapping $\fwm$
\cite{byPronySing12}, will eventually provide a satisfactory answer
to this question.

As we discuss in \cite{batenkov_numerical_2014}, decimation is related
to other similar ideas in numerical analysis \cite{sidi_practical_2003}
and signal processing \cite{kia_high-resolution_2007,kim_high-resolution_2013,maravic_sampling_2005}.
In symbolic-numeric literature connected with sparse numerical polynomial
interpolation (i.e. in the noisy setting), the possible ill-conditioning
of the Hankel matrices $H_{d}$ can be overcome either by random sampling
of the nodes $\left\{ \nd_{j}\right\} $ \cite{giesbrecht_symbolicnumeric_2009,kaltofen_early_2003,kaltofen_fast_2012}
or by the recently introduced affine sub-sequence approach \cite{kaltofen_sparse_2014}
for outlier detection (see also \cite{comer_sparse_2012}), which
is in many ways similar to decimation.

It would be interesting to establish more precise connections of our
method to these works.

\section*{References}

\bibliographystyle{elsarticle-harv}
\phantomsection\addcontentsline{toc}{section}{\refname}\bibliography{tcs14}

\end{document}